\documentclass[a4paper]{amsart}

\usepackage{amssymb}
\usepackage{hyperref}

\title{A note on two-colorability of nonuniform hypergraphs}

\author{Lech Duraj}
\author{Grzegorz Gutowski}
\author{Jakub Kozik}
\address{Theoretical Computer Science Department, Faculty of Mathematics and Computer Science, Jagiellonian University, Krak\'{o}w, Poland}
\email{\{Lech.Duraj, Grzegorz.Gutowski, Jakub.Kozik\}@uj.edu.pl}

\thanks{This work was partially supported by Polish National Science Center (2016/21/B/ST6/02165)}


\newtheorem{theorem}{Theorem}[section]
\newtheorem{lemma}[theorem]{Lemma}
\newtheorem{prop}[theorem]{Proposition}

\renewcommand{\epsilon}{\varepsilon}

\renewcommand{\leq}{\leqslant}
\renewcommand{\geq}{\geqslant}

\newcommand{\expt}{\mathbb{E}}

\newcommand{\prob}{\mathrm{Pr}}
\newcommand{\probG}{\mathrm{Pr}_\mathcal{G}}

\newcommand{\minedge}{s_{\min}}

\begin{document}

\begin{abstract}
For a hypergraph $H$, let $q(H)$ denote the expected number of monochromatic edges when the color of each vertex in $H$ is sampled uniformly at random from the set of size 2.
Let $\minedge(H)$ denote the minimum size of an edge in $H$.
Erd\H{o}s asked in 1963 whether there exists an unbounded function $g(k)$ such that any hypergraph $H$ with $\minedge(H) \geq k$ and $q(H) \leq g(k)$ is two colorable.
Beck in 1978 answered this question in the affirmative for a function $g(k) = \Theta(\log^* k)$.
We improve this result by showing that, for an absolute constant $\delta>0$, a version of random greedy coloring procedure is likely to find a proper two coloring for any hypergraph $H$ with $\minedge(H) \geq k$ and $q(H) \leq \delta \cdot \log k$.
\end{abstract}

\maketitle

\newcommand{\ca}{\alpha_\mathcal{A}}
\newcommand{\cc}{\alpha_\mathcal{C}}
\newcommand{\cd}{\alpha_\mathcal{D}}

\newcommand{\rg}{\mathcal{R}_3}

\newcommand{\cp}{\alpha_\mathcal{B}}

\section{Introduction}
A \emph{hypergraph} $H=(V,E)$ is a finite set of vertices $V$ and a set of edges $E$ where each edge is a set of at least two vertices.
A \emph{two coloring} of $H$ is an assignment of color blue or red to each vertex in $H$.
A coloring is \emph{proper} if each edge in $H$ contains both a vertex colored blue and a vertex colored red.
We say that $H$ is \emph{two colorable} if it admits a proper two coloring.
Hypergraph $H$ is \emph{$k$-uniform} if every edge in $H$ has size exactly $k$ -- we also say that $H$ is a \emph{$k$-graph}.
For every $n \in \mathbb{N}$, the set $\{1,\ldots,n\}$ is denoted by $[n]$.
We use standard $O$-notation to describe asymptotic properties of various functions.

One of the most classical problems in the extremal combinatorics is to find the minimum number of edges $m(k)$ in a $k$-uniform hypergraph that is not two colorable.
The research on this problem has been started in the 60s by Erd\H{o}s and Hajnal \cite{ErdHaj61}, who used the term \emph{Property~B} for two colorability.
Today, by the result of Radhakrishnan and Srinivasan \cite{RS00}, we know that $m(k) = \Omega((k/\log k)^{1/2})\cdot 2^{k}$.
The best known upper bound, proved by Erd\H{o}s \cite{Erd1964} in 1964, is $m(k)= O(k^2)\cdot 2^{k}$.
This upper bound follows from the fact that a random $k$-graph with $k^2$ vertices and $O(k^2)\cdot 2^k$ edges is very unlikely to be two colorable.
Interestingly, known deterministic constructions require much larger structures -- the best one is by Gebauer \cite{Gebauer13} and gives a not two colorable $k$-graph with roughly $2^{k+k^{2/3}}$ edges.

Lov\'{a}sz \cite{Lov73} proved that for $k\geq 3$, the problem of deciding if a given $k$-graph is two colorable is NP-complete.
For $k$-graphs with the number of edges smaller than $m(k)$ the decission problem is trivial -- by the definition they are all two colorable.
Nevertheless, constructing a two coloring of such $k$-graphs is not necessarily an easy task.
Luckily, the known lower bounds for $m(k)$ are constructive.
In fact, the bound of \cite{RS00} is proved by showing that some randomized coloring procedure succeeds with high probability for the considered hypergraphs.
Cherkashin and Kozik \cite{CK15} showed that the same bound is obtained by the analysis of a random greedy algorithm
(i.e., a procedure that colors the vertices of a hypergraph in a random order and assigns color blue to each vertex unless it is the last vertex of a monochromatic blue edge -- only then color red is assigned).


 
For a hypergraph $H=(V,E)$, let $q(H)$ denote the expected number of monochromatic edges when the color of each vertex is sampled uniformly at random.
Clearly, for a $k$-graph $H$, we have $q(H)= |E| \cdot 2^{-k+1}$, and determining the value of $m(k)$ is equivalent to finding a not two colorable $k$-graph $H$ with the minimal possible value of $q(H)$.
This formulation allows for a neat generalization of the question to hypergraphs with edges of arbitrary sizes (i.e., nonuniform hypergraphs).
For a hypergraph $H=(V,E)$, let $\minedge(H) = \min_{e \in E} |e|$ and observe that $q(H) = \sum_{e\in E} 2^{-|e|+1}$.
Erd\H{o}s \cite{Erd1963, EL1975} asked whether there exists an unbounded function $g$ such that any hypergraph $H$ with $\minedge(H) \geq k$ and $q(H)\leq g(k)$ is two colorable.
A positive answer has been given in 1978 by Beck \cite{Beck78} who proved the result for $g(k)= \Theta(\log^*(k))$.
This has not been improved since then.
(In 2008 L. Lu announced a proof of a bound $\Omega(\log(k)/\log\log(k)$ but it turned out to work only for simple hypergraphs.
Shabanov in \cite{Shab15} improved the bound for this class to $\Omega(\sqrt{k})$.)
In this paper we prove the same result for $g(k) = \Theta(\log(k))$.
The random construction of a not two colorable $k$-graph by Erd\H{o}s \cite{Erd1964} shows that the best possible $g$ is $O(k^2)$, even when restricted to uniform hypergraphs.
Interestingly, there are no better nonuniform constructions known.
Our main result is the following.

\begin{theorem}
\label{thm:main}
    There exists a constant $\delta>0$ such that for all sufficiently large $k$, any hypergraph $H=(V,E)$ with $\minedge(H) \geq k$ and $q(H) \leq \delta \cdot \log k$ is two colorable.
\end{theorem}
Moreover, we prove the theorem by showing that a version of a random greedy coloring procedure succeeds with positive probability for these hypergraphs.

\section{Basic notions and the coloring procedure}

\subsection{Tools}

We start with a simple lemma on convex functions of random variables. 

\begin{lemma}
\label{lemma:conv-expt}
Let $X$ be a nonnegative random variable such that $0 \leq X \leq M$ and $\expt[X] \leq \lambda M$ for some $M \geq 0$, and $0 \leq \lambda \leq 1$. 
Then, for any convex function $f:[0, M] \to [0, \infty)$ with $f(M) \geq f(0)$, the following inequality holds
 
    \[
    \expt[f(X)] \leq \lambda f(M) + (1 - \lambda) f(0)\text{.} 
    \]
 
\end{lemma}
\begin{proof}
Consider another random variable $Y := \frac{X}{M} f(M) + (1 - \frac{X}{M}) f(0)$. From the convexity of $f$ we have $f(X) \leq Y$. Therefore
\[
\expt[f(X)] \leq \frac{\expt[X]}{M} \cdot f(M) + f(0) - \frac{\expt[X]}{M} \cdot f(0) \leq f(0) + (f(M) - f(0)) \cdot \frac{\expt[X]}{M} \leq \lambda f(M) + (1 - \lambda) f(0)\text{,}
\]
as desired.
\end{proof}

\subsection{Preliminaries}
Let $H=(V,E)$ be a hypergraph and let $k$ denote the minimum size of an edge in $H$.
For any $j\geq k$ we define
\[
	q_j := \sum_{e \in E, |e|=j} 2^{-j+1},
\]
which is the expected number of monochromatic edges of size $j$ when the color of each vertex is sampled uniformly at random.
Let $q:= q(H)$ and observe that $q = \sum_{j\geq k} q_j$.

We aim to prove that if $q = O(\log k)$ then the hypergraph is two colorable.
In order to do that, we describe a random coloring procedure and with a careful analysis we bound the probability that a fixed edge is monochromatic after the procedure finishes.
The obtained bound allows us to conclude that the expected number of monochromatic edges after the procedure finishes is smaller than one.
Thus, the hypergraph is two colorable.

\subsection{The coloring procedure}

Our algorithm is based on the \emph{random greedy coloring} and it works in two phases:

\begin{enumerate}
\item \textbf{Initial coloring}

We start by independently sampling, for every vertex $v$, the following two values:
\begin{itemize}
	\item $ic(v)$ -- the \emph{initial color} of $v$: blue or red, each with probability $\frac{1}{2}$,
	\item $w(v)$ -- the \emph{weight} of $v$, sampled uniformly at random from the real interval $(0,1)$.
\end{itemize}
The edge in which all vertices get the same initial color is called \emph{initially monochromatic}.
We assume that no two vertices have the same weight -- we discuss this in Section \ref{sec:global-events}.
For an edge $e$, the \emph{heaviest} vertex in $e$ is the one with maximum weight among vertices in $e$.
We define the weight $w(e)$ of $e$ to be the weight of the heaviest vertex in $e$ (i.e., $w(e)= \max_{v\in e} w(v)$).

\item \textbf{Recoloring}

We iterate over all vertices in the order of increasing weight and for each vertex $v$, we define $c(v)$ -- the \emph{color} of $v$.
Once the color is assigned, it is never changed.
We say that a vertex $v$ is \emph{recolored} if $c(v)$ is already defined and it is different than $ic(v)$.
Our goal is to recolor at least one vertex in each initially monochromatic edge.
If $v$ is the heaviest vertex in some initially monochromatic edge $e$, then the color of every other vertex in $e$ is already defined.
If none of the vertices in $e$ is recolored, we define $c(v)$ to be the color other than $ic(v)$ (i.e., we recolor $v$), and we say that $e$ is a \emph{reason} to recolor $v$.
Note that there may be more than one reason to recolor $v$.
If there is no reason to recolor $v$ (i.e., no initially monochromatic edge with no recolored vertex and heaviest vertex $v$) we simply assign $c(v) = ic(v)$.
Observe that eventually every initially monochromatic edge gets one of the vertices recolored.
\end{enumerate}

\subsection{Main result}

For a better exposition of the argument, we first prove a statement slightly weaker than Theorem \ref{thm:main}.
In Section \ref{sec:simple-bound} we give a proof of the following result about the coloring procedure.
\begin{prop}
\label{prop:oneEdgeBound-simple}
 If $q = O(\frac{\log k}{\log \log k})$ then, for any edge $e$, the probability that all vertices in $e$ are colored red does not exceed $\frac{1}{3q2^{|e|-1}}$.
\end{prop}
This immediately implies that the expected number of monochromatic edges is at most $2 \cdot \sum_{e \in E} \frac{1}{3\cdot q \cdot 2^{|e|-1}} = \frac{2}{3}$ and thus, not only $H$ is two colorable but also that our coloring procedure succeeds with probability at least $\frac{1}{3}$.
In Section \ref{sec:better-bound} we introduce more technical details to the argument and improve the bound.
\begin{prop}
\label{prop:oneEdgeBound-better}
 If $q = O(\log k)$ then, for any edge $e$, the probability that all vertices in $e$ are colored red does not exceed $\frac{1}{3q2^{|e|-1}}$.
\end{prop}
This immediately implies Theorem \ref{thm:main}.

\section{Analysis}

\subsection{Bad events}
\label{sec:global-events}

The proof focuses on bounding the probability that one, fixed edge becomes monochromatic red.
Nevertheless, we want to first exclude some problematic but unlikely events from happening.
The simplest example is that we don't want two different vertices to receive the same weight.
The probability of this event is zero and we want to simply assume that it doesn't happen.
To be more precise, we allow our coloring procedure to fail during the initial phase of coloring.
We give a few different reasons to fail and we argue that the probability that any of those bad events happens is small.
Then, for the rest of the proof, we assume that none of the bad events happens.

\subsubsection{Event $\mathcal{A}$ -- too many initially monochromatic edges}
The expected number of initially monochromatic edges is $q$.
For a constant $\ca$ (to be fixed later) let $\mathcal{A}$ denote the event that there are more than $\ca \cdot q$ initially monochromatic edges.
Markov inequality gives that $\prob[\mathcal{A}] < 1/\ca$. 

\subsubsection{Event $\mathcal{B}$ -- a light monochromatic edge}
For a constant $\cp$ (to be fixed later) and every $j$ we define 
\[
	p_j := \frac{\ln (\cp q)}{j}.
\]
An edge $e$ of size $j$ is \emph{light} if it is initially monochromatic and the weight of every vertex in $e$ is smaller than $1-p_j$.
The expected number of light monochromatic edges of size $j$ is 
\[
    q_j 2^{j-1} \cdot (1-p_j)^j \cdot 2^{-j+1} < q_j  \cdot j \cdot \exp(-p_j) = \frac{q_j}{\cp q}.
\]
Therefore, the expected total number of light edges (of any size) is at most $1/\cp$.
Let $\mathcal{B}$ denote the event that there is a light monochromatic edge.
Clearly $\prob[\mathcal{B}]< 1/\cp$. 

\subsubsection{Event $\mathcal{C}$ -- too many almost monochromatic edges}
An edge $e$ is \emph{almost monochromatic} if there is a vertex $v\in f$ such that all vertices in $f-v$ have the same initial color (in particular, an initially monochromatic edge is also an almost monochromatic edge).
With every almost monochromatic edge $f$ we can injectively associate a \emph{certifying pair} $(f,v)\in E\times V$ for which $v\in f$ and $f - v$ is initially monochromatic.

Let $Q_j$ be a random variable that denotes the number of almost monochromatic edges of size $j$.
Since the number of such edges cannot exceed the number of certifying pairs associated with edges of size $j$, we get
$\expt[Q_j] \leq q_j 2^{j-1} \cdot j \cdot 2^{-j+2} = 2 j \cdot q_j$.
We define random variable
\[
	Y:= \sum_j \frac{Q_j}{j},
\]
and get that $\expt[Y]\leq 2q$.
Let $\mathcal{C}$ denote the event that $Y > \cc q$.
Markov inequality gives $\prob[\mathcal{C}]< 2/\cc$.
\bigskip
\\
For any fixed $\epsilon>0$ we can choose constants $\ca,\cp,\cc$ so that $1/\ca + 1/\cp+ 2/\cc  < \epsilon$.
Denote by $\mathcal{G}$ the intersection $\overline{\mathcal{A}} \cap \overline{\mathcal{B}} \cap \overline{\mathcal{C}}$ and observe that $\prob[\mathcal{G}] >1-\epsilon$.
That is, with arbitrarily high probability none of the bad events happens.
For any event $\mathcal{V}$, we denote $\prob[\mathcal{V} \cap \mathcal{G}]$ by $\probG[\mathcal{V}]$
and similarily by $\probG[\mathcal{V} | \mathcal{C}]$ we mean $\prob[\mathcal{V} \cap \mathcal{G}| \mathcal{C}]$.

\subsection{$e$-focused coloring}
For the rest of this section and the next section we fix an arbitrary edge $e$ in $E$.
Let $s$ denote the size of $e$.
The event ``\emph{$e$ becomes red}'' denotes the situation that all vertices in $e$ are colored red by the coloring procedure.
First observation is that if $e$ is initially monochromatic red, then at least one vertex in $e$ gets recolored and $e$ can't become red in the end.
Thus, if $e$ becomes red then $e$ contains some initially blue vertices and each of them gets recolored.
In particular, every initially blue vertex in $e$ is the heaviest vertex in some initially monochromatic blue edge.
Additionally, it needs to happen that none of the initially red vertices in $e$ gets recolored, but this condition seems impossible to use.

Taking into account the bad events we aim to prove that for a proper $q$ we have:
\[
	\probG[\text{$e$ becomes red}]< \frac{1}{3} \cdot \frac{1}{q 2^{s-1}}.
\]

\subsubsection{The threat hypergraph}
In what follows we try to understand better which initially blue vertices in $e$ are recolored to red.
The important observation is that if edges $f$ and $e$ have more than one vertex in common and $f$ is a reason to recolor any of the common vertices, then $e$ does not become red.
To see that, let $v$ be the heaviest vertex in $f$, and let $w$ be any vertex in $f \cap e$ other than $v$.
If $f$ is a reason to recolor $v$ then $f$ is initially monochromatic blue and $w$ is not recolored.
Thus, $w$ retains the initial blue color, and edge $e$ does not become red.

This motivates the following construction of the \emph{threat hypergraph} $H_e$.
We define the vertex set of $H_e$ to be $V\setminus e$.
For each edge $f$ in $E$ that has exactly one common vertex with $e$ (i.e., $|f \cap e| = 1$), let $f_e = f \setminus e$.
We define the edge set of $H_e$ to be $\{f_e : f \in E, |f \cap e| = 1\}$.
Observe that for different edges $f \neq f'$ in $E$ it might happen that $f_e = f'_e$.
Thus, $H_e$ is a multihypergraph.
For each edge $f_e$ of $H_e$ we call $f$ to be the \emph{extension edge} of $f_e$ and we call the only vertex in $f \cap e$ to be the \emph{extension vertex} of $f_e$.

For the sake of our analysis, we reveal the outcomes of the random experiments used in the coloring procedure in four steps. 
In the first step we reveal the initial colors of the vertices in $H_e$.
In the second step, we reveal the initial colors of the vertices in $e$.
Then, we reveal the weights of vertices in $H_e$.
Finally, we reveal the weights of the vertices in $e$.
It is crucial to understand that this does not influence the coloring procedure in any way.

After the first step, some edges in $H_e$ are monochromatic blue.
For every such an edge $f_e$, let $v$ be the extension vertex of $f_e$, and we say that $v$ is \emph{endangered} by $f_e$.
Observe that if $e$ is to become red, then among vertices in $e$, only the endangered ones can be recolored from blue to red.
For every endangered vertex $v$ in $e$ we define the \emph{severity} of $v$ to be the minimum size $|f|$ of an edge such that $v$ is endangered by $f_e$.
We define $\mathcal{R}^e_j$ to be the set of all vertices in $e$ that are endangered and with severity $j$.
Let $R_j^e:=|\mathcal{R}^e_j|$. 
Note that both $R_j^e$ and $\mathcal{R}^e_j$ are random variables which are determined after the first step (i.e., by the initial colors of the vertices in $H_e$).

Thus, a necessary condition for $e$ to become red is that in the second step only the endangered vertices get initial color blue.
Consider an endangered vertex with severity $j$ which is initially blue and which is to become red.
There is an edge $f_e$ that endangers $v$, and $v$ becomes the heaviest vertex in the extension of $f_e$.
In particular, since the size of $f$ is at least $j$, the weight of $v$ (revealed in the fourth step) has to be at least $1-p_j$.
Otherwise, the edge $f$ is a light monochromatic edge and bad event $\mathcal{B}$ happens.

Observe that there are no more vertices recolored than there are initially monochromatic edges.
As we assume that bad event $\mathcal{A}$ does not happen, there are at most $\ca q$ vertices recolored in total.
Let us sum up the observed necessary conditions for the edge $e$ to become red:
\begin{enumerate}
	\item at least one and at most $\ca q$ vertices in $e$ are initially blue,
    \item every initially blue vertex $v$ in $e$ is endangered. If severity of $v$ is $j$ then $w(v) \geq 1-p_j$.
\end{enumerate}
We use these conditions to obtain an upper bound on the probability of $e$ becoming red.

\subsection{Simple bound}
\label{sec:simple-bound}
We define, mainly for technical convenience, a random variable
\[
    X:=\sum_j  R_j^e \cdot p_j\text{.}
\]
Observe that $X$ is determined after the first step and that $X$ takes only a finite number of possible values.
For the rest of this section whenever we condition on event $X = x$ we always assume that the value $x$ is such that $\prob[X=x]>0$.
The bound will follow from the following result:
\begin{prop}
\label{prop:bound1-4}
\[
	\probG[\text{$e$ becomes red } | X=x] < \frac{\exp(x)-1}{2^{s}}.
\]
\end{prop}
\begin{proof}
Assume that we are after the first step and the values of variables $\mathcal{R}^e_j$, $R_j^e$, and $X$ are determined.
For each $j \geq k$, let $r_j := R_j^e$.
With this assumption, we compute the probability of $e$ becoming red.
We claim that
\[
    \probG[\text{$e$ becomes red } | \text{the first step} ]
	\leq 	\frac{1}{2^{s-\sum_j r_j}}
	\sum_{1\leq c_k+c_{k+1}+ \ldots \leq \ca q} 
	\prod_{j} 
	{r_j \choose c_j} 
	\left( \frac{p_j}{2} \right)^{c_j}
	\left( \frac{1}{2} \right)^{r_j-c_j}.
\]
The first factor corresponds to the not endangered vertices in $e$ -- each of them needs to be initially colored red.
The sum spans over the values $c_k, c_{k+1}, \ldots$, where $c_j$ corresponds to the number of initially blue vertices in $R_j^e$.
There are exactly $\sum_{j} c_j$ initially blue elements and we know that this number is at least $1$ and at most $\ca q$.

Once the number of initially blue elements in each $\mathcal{R}^e_j$ is fixed, there are ${r_j \choose c_j}$ possibilities to choose these elements from $\mathcal{R}^e_j$. 
Finally, all the chosen elements have to be initially colored blue and their weight has to be at least $1-p_j$.
The remaining elements of $\mathcal{R}^e_j$ have to be initially colored red.

Observe that the expression depends not on a particular result of the first phase, but rather only on the values of $R_j^e$.
We use the fact that ${r_j \choose c_j} \leq \frac{r_j^{c_j}}{c_j!}$, rearrange the terms, and obtain:
\begin{align*}
	\probG[\text{$e$ becomes red } | (R_j^e = r_j)_{j \geq k}]
	& \leq
	\frac{1}{2^{s}}
	\sum_{c=1}^{\ca q}
	\sum_{c_k+c_{k+1}+ \ldots =c} 
	\prod_{j} 
	{r_j \choose c_j} 
		p_j^{c_j}
\\
& \leq  \frac{1}{2^{s}}
	\sum_{c=1}^{\ca q}
	\frac{1}{c!}
    \left( \sum_j  r_j \cdot p_j\right)^c\text{.}
\end{align*}

Let $x := \sum_j r_j \cdot p_j$, recall that $X=\sum_j  R_j^e \cdot p_j$, and observe that the last expression depends not on the particular values of $R_j^e$, but rather only on the value of $X$.
\begin{align}
\label{eq:BoundBySum}
	\probG[\text{$e$ becomes red } | X=x] 
	& \leq \frac{1}{2^{s}}
	\sum_{c=1}^{\ca q}
	\frac{x^c}{c!}\\
\label{eq:BoundByExp}	&< 
    \frac{\exp(x)-1}{2^{s}}\text{.}
\end{align}
\end{proof}
\begin{proof}[Proof of Proposition \ref{prop:oneEdgeBound-simple}]
Recall that the random variable $Q_j$ denotes the number of almost monochromatic edges of size $j$, while $R_j^e$ is the number of endangered vertices in $e$ with severity $j$.
For every such a vertex $v$ we have an initially blue edge $f$ in $H_e$ for which $v$ is the extension vertex.
Since the extension edge of $f$ is almost blue we obtain that $R_j^e \leq Q_j$.
This implies:
\begin{equation}
\label{eq:XlogY}
	X= \sum_{j} R_j^e \cdot p_j = \ln(\cp q) \cdot \sum_{j} \frac{R_j^e}{j} \leq \ln(\cp q) \cdot \sum_{j} \frac{Q_j}{j} =    
     \ln(\cp q) \cdot Y\text{.}
\end{equation}
Therefore $X\leq \ln(\cp q) \cdot  \cc q$ unless bad event $\mathcal{C}$ happens.
We now have:
\begin{align*}
	\probG[\text{$e$ becomes red}] 
	& = \sum_{x\leq \ln(\cp q) \cdot  \cc q} \prob[X=x] \cdot \probG[\text{$e$ becomes red } | X=x]  \\
	& < \sum_{x\leq \ln(\cp q) \cdot  \cc q} \prob[X=x] \cdot \frac{\exp(x)-1}{2^{s}}  \\
    & \leq \frac{1}{2^{s}} \cdot \expt\left[\exp(X)-1 \Big| X\leq \ln(\cp q) \cdot  \cc q \right]\text{.}
\end{align*}
Inequality (\ref{eq:XlogY}) also yields
\[
    \expt[X] \leq \sum_j q_j \cdot p_j \leq \frac{q \ln(\cp q)}{k}\text{.}
\]
We apply Lemma \ref{lemma:conv-expt} for $f(x) = \exp(x) - 1$, $M = \cc q \ln(\cp q)$ and $\lambda = \frac{1}{\cc k}$, and obtain:
\[
    \expt[\exp(X)-1] \leq \frac{\exp(\ln(\cp q) \cdot  \cc q)}{\cc k}\text{.}
\]
Hence
\[
    \probG[\text{$e$ becomes red}]  < \frac{1}{2^{s}} \cdot \frac{\exp(\ln(\cp q) \cdot  \cc q)}{\cc k}\text{.}
\]
Let $\alpha > \max \{\cp,\cc\}$. Now, suppose that:
\[
    q \leq \frac{1}{\alpha}\cdot \frac{\ln k }{\ln\ln k}\text{.}
\]
For $k$ large enough, $\ln(\cp q) \leq \ln\ln(k)$ which yields:
\[
\frac{\exp(\ln(\cp q) \cdot  \cc q)}{\cc k} 
\leq \frac{1}{\cc k}\cdot \exp\left(\frac{\cc \ln k}{\alpha \ln \ln k} \cdot \ln \ln k \right) 
\leq \frac{k^{(\cc/\alpha) -1}}{\cc}\text{.}
\]
For $k$ large enough, the last term is less than $\frac{1}{6q}$, which implies:
\[
    \probG[\text{$e$ becomes red}]  < \frac{1}{3q \cdot 2^{s-1}}
\]
and completes the proof.
\end{proof}
An astute reader may have realised that we did not use bad event $\mathcal{A}$ in any essential way. 
Currently, the only reason to introduce $\mathcal{A}$ is that it makes the proof slightly easier.
We could, however, use $\mathcal{A}$ to improve bound (\ref{eq:BoundByExp}) for the values of $x$ greater than $q$, leading to
a slightly better condition $q = O(\frac{\log k}{\log\log\log k})$. 
We do not elaborate on that since the argument in Section \ref{sec:better-bound} already gives an even better result.

\subsection{Improved bound}
\label{sec:better-bound}

In order to obtain an improved bound we introduce one more bad event.

\subsubsection{Event $\mathcal{D}$ -- large second weight deficit}
For every edge $f$ in $E$ which is initially monochromatic, we define its \emph{second weight deficit} as $d_2(f):=(|f|+1) \cdot (1- w_2(f))$, where $w_2(f)$ is the weight of the second heaviest vertex in $f$.
For an edge $f$ that is not initially monochromatic, $d_2(f)$ is defined to be 0.

Note that, conditioned on $f$ being initially monochromatic, the variable $1-w_2(f)$ has mean $\frac{2}{|f|+1}$.
In particular $\expt[d_2(f)| \text{ $f$ is monochromatic}]=2$ and hence $\expt[d_2(f)] = \frac{2}{2^{|f|-1}}$.
Let $D_2:= \sum_{f\in E} d_2(f)$ and observe that we have
\[
	\expt[D_2] = \sum_{j} 2q_j = 2q.
\]
Event $\mathcal{D}$ is defined as $D_2 > \cd q$. 
By Markov inequality, we get $\prob[\mathcal{D}]< 2/\cd$ and we can chose $\cd$ so that this probability is arbitrarily small.

%
%
\subsubsection{Analysis}
In the first step of $e$-focused coloring we reveal the initial colors of all the vertices from $V\setminus e$. 
This step determines the endangered vertices in $e$ -- we denote their set by $\mathcal{R}$.
For every value of $c = 1, 2, \ldots \ca q$ and every $c$-subset $S=\{v_{1}, \ldots, v_{c}\}$ of $\mathcal{R}$ we consider an event that $S$ contains exactly the vertices in $e$ which become recolored.
Thus, these are the only initially blue vertices in $e$.
Once we fix the subset $S$, the probability that $S$ is the set of initially blue vertices in $e$ is precisely $2^{-s}$.
This event is determined after the second step of $e$-focused coloring -- when the initial colors of vertices in $e$ are revealed.
In order to be recolored, every vertex $v_{j}$ must receive a weight that makes it heavier than some edge that endangers it.
Let us reveal the weights of the vertices in $V\setminus e$ (third step of $e$-focused coloring).
The vertex $v_{j}$ is endangered by some edges $f_{v_j}^1, \ldots, f_{v_j}^{t}$ of $H_e$, and let $f_{v_j}$ be the lightest of these edges 
(i.e. the edge whose heaviest vertex is the lightest among the heaviest vertices of $f_{v_j}^1, \ldots, f_{v_j}^t$).
Clearly in order for vertex $v_{j}$ to be recolored, it has to get a weight greater than $w(f_{v_j})$
-- this happens with probability $1-w(f_{v_j})$.
We choose a parametrization that takes into account the size of $f_{v_j}$ and denote the value $1-w(f_{v_j})$ by $\frac{\delta_j}{|f_{v_j}|+2}$.
Now, conditioned on the result of the first three steps, the probability that all vertices $\{v_{1}, \ldots, v_{c}\}$ are heavy enough is
\begin{equation}
\label{eq:fdelta}
    \prod_{j=1}^c \frac{\delta_j}{|f_{v_j}|+2} < \prod_{j=1}^c \frac{\delta_j}{|f_{v_j}|+1} \text{.}
\end{equation}
The edge $f_{v_j}$ together with $v_j$ forms an edge of $H$, which we denote by $h_{v_j}$.
Although the value of $d_2(h_{v_j})$ is not determined until we reveal the weight of $v_{j}$ (in the fourth step), we already know at this point that $\delta_j \leq d_2(h_{v_j})$ (it becomes an equality when $v_{j}$ becomes the heaviest vertex in $h_{v_j}$).
Assuming the bad event $\mathcal{D}$ does not happen, we have $\sum_{j=1}^c \delta_j \leq \cd q$.
Using the AM-GM inequality we deduce that $\prod_{j=1}^c \delta_j \leq \left( \frac{\cd q}{c}\right)^c$, which bounds the value of (\ref{eq:fdelta}):
\[
	\prod_{j=1}^c \frac{\delta_j}{|f_{v_j}|+1} \leq \left( \frac{\cd q}{c}\right)^c \prod_{j=1}^c \frac{1}{|f_{v_j}|+1}.	
\]

Summing over all $c$-subsets of $\mathcal{R}$ we get that the probability that some $c$-subset contains all initially blue vertices in $e$ and they are all recolored does not exceed
\begin{align*}
	\sum_{S \in {\mathcal{R} \choose c}}\frac{1}{2^{s}} \left( \frac{\cd q}{c}\right)^c \prod_{v\in S} \frac{1}{|f_v|+1}
	\leq \frac{1}{2^{s}} \left( \frac{\cd q}{c}\right)^c \frac{1}{c!}\left( \sum_{v\in \mathcal{R}} \frac{1}{|f_v| +1} \right)^c
.
\end{align*}

Define random variable
\[
    Y_e:= \sum_{\text{$f_e$ in $H_e$, $f_e$ is blue}} \frac{1}{ |f_e|+1}\text{,}
\]
which gets determined after the first step of $e$-focused coloring.
For each endangered vertex $v$ in $e$, all the edges, including the lightest one, that endanger $v$ are blue and thus are taken into the sum defining $Y_e$.
Therefore, $Y_e \geq \sum_{v\in \mathcal{R}} \frac{1}{|f_v| +1}$.
On the other hand, the extension edge of every blue edge in $H_e$ is an almost monochromatic edge in $H$.
As $Y$ counts the number of almost monochromatic edges in $H$, we get $Y_e \leq Y \leq \cc q$ unless bad event $\mathcal{C}$ happens.
We can also bound the expected value of $Y_e$:
\[
	\expt[Y_e] 
	= \sum_{f\in H_e} \frac{2^{-|f|}}{|f|+1} 
	< \frac{1}{k} \sum_{f'\in H} 2^{-|f'|+1}
    = \frac{q}{k}\text{.}
\]
Note that $Y_e$ takes only a finite number of possible values.
For any value $y$ such that $\prob[Y_e=y]>0$, we get the following bound:
\begin{align*}
	\probG[\text{$e$ becomes red } | Y_e=y] 
	&\leq
		\frac{1}{2^{s}} \sum_{c=1}^{\ca q} \left( \frac{\cd q}{c}\right)^c \frac{y^c }{c!}
\\
	&\leq
		 \frac{1}{2^{s}} \sum_{c=1}^{\ca q} \frac{(\cd q y)^c}{c! \cdot c^c}
\\
	& \leq
		 \frac{1}{2^{s}} \sum_{c=1}^{\ca q} \frac{(2\cd q y)^c} {(2c)!}	 
\end{align*}
as $\frac{(2c)!}{2^c} \leq \frac{c! \cdot (2c)^c}{2^c} = c! \cdot c^c$.
For any $x$, we have $\sum_{c=0}^{\infty} \frac{x^{2c}}{(2c)!} = \frac{\exp(x) + \exp(-x)}{2} = \cosh(x)$.
Therefore
\[
    \probG[\text{$e$ becomes red } | Y_e=y] \leq \frac{1}{2^{s}} (\cosh\left(\sqrt{2\cd q y}\right) - 1)\text{,}
\]
and
\[
	\probG[\text{$e$ becomes red}] 
	\leq 
    \expt\left[ \frac{1}{2^{s}} (\cosh\left(\sqrt{2\cd q Y_e}\right) - 1) \Big| Y_e \leq \cc q\right]\text{.}
\]

Observe that for any $a>0$, the function $\cosh\left(a \sqrt{x}\right)$ is convex and increasing in $[0, \infty)$.
Therefore, 
we apply Lemma \ref{lemma:conv-expt} for $f(x) = \cosh\left(\sqrt{2\cd q x}\right) -1$, $M = \cc q$ and $\lambda = \frac{1}{\cc k}$, and obtain:
\[
	\probG[\text{$e$ becomes red}] 
	\leq
\frac{1}{2^{s}} \frac{1}{\cc k}(\cosh\left(\sqrt{2 \cd \cc q q}\right) - 1) \leq 
\frac{1}{\cc k 2^{s}} \exp\left(\sqrt{2\cd \cc} \cdot q\right)
    \text{.}
\]  
The obtained value is smaller than $\frac{1}{3 q 2^{s-1}}$ whenever
\[
    \frac{3 q \exp\left(\sqrt{2 \cd \cc} \cdot q\right)}{2 \cc k} \leq 1\text{.}
\]
The last inequality is easily seen to hold for $q\leq \frac{0.9}{\sqrt{2 \cd \cc}} \ln k$ and all large enough $k$.

\section{Remarks}
\subsection{Bounded maximal size}
We can derive better bounds when the size of the maximum edge is not much larger than $k$.
Suppose that $\max_{e\in E} |e| \leq K$.
We apply the proof strategy from \cite{CK15} and analyze the random greedy coloring procedure (i.e. we arrange the vertices in random order and color consecutive vertices blue if this does not create a monochromatic edge, otherwise we color it red).
As a technical convenience, instead of sampling a random ordering of the vertices, for every vertex we choose uniformly a weight from the real interval $(0,1)$.
We color vertices greedily in the order of increasing weights.
We choose (with foresight) parameter $p:= \ln(4q)/k$.
An edge is called \emph{light} if the weight of its heaviest vertex is at most $(1-p)/2$.
Similarly an edge $f$ is \emph{heavy} if every verex in $f$ has weight at least $(1+p)/2$.
The probability that there exists a light edge is bounded by the expected number of such:
\[
	\sum_{f\in E} \left( \frac{1-p}{2}\right)^{|f|} \leq (1-p)^k \cdot q.
\] 
The same bound holds for \emph{heavy} edges.
It is easy to see that in order for the procedure to fail there must exist a pair of edges $f_1, f_2$ such that the heaviest vertex of $f_1$ is the lightest vertex of $f_2$. 
Such a pair is called \emph{conflicting}.
Therefore for the procedure to fail it is necessary that either there exists a conflicting pair $f_1,f_2$ for which the weight of the unique common vertex belongs to $((1-p)/2, (1+p)/2)$ or there exists a light or heavy edge.
The expected number of such conflicting pairs is at most
\begin{align*}
&  \sum_{f_1, f_2 \in E} \int_{-p/2}^{p/2} \left(\frac{1}{2}+x\right)^{|f_1|-1} \left(\frac{1}{2}-x\right)^{|f_2|-1} dx
\\
	&\leq
	\sum_{f_1, f_2 \in E} 2^{-|f_1|-|f_2|+2} \cdot p \cdot \max_{x\in (-p/2, p/2)} (1+2x)^{|f_1|-1}(1-2x)^{|f_2|-1} 
\\
	&\leq
	p \cdot \sum_{f_1, f_2 \in E} 2^{-|f_1|-|f_2|+2} \cdot \max_{x\in (-p/2, p/2)} (1+2x)^{|f_1|-|f_2|}
\\
	&=
		p \cdot \sum_{f_1, f_2 \in E} 2^{-|f_1|-|f_2|+2} \cdot (1+p)^{|f_1|-|f_2|}
\\
	&\leq
		p \cdot (1+p)^{K-k} \sum_{f_1, f_2 \in E} 2^{-|f_1|-|f_2|+2} = p (1+p)^{K-k} q^2.
\end{align*}
Altogether the probability of failure is at most
\[
	p (1+p)^{K-k} q^2 + 2 (1-p)^k q  \sim p q^2 \exp(p(K-k)) +2 q \exp(-p k).
\]
Plugging in the value of $p$ we get
\[
	\frac{\ln(4q) q^2 (4q)^{K/k-1}}{k} + 1/2 \leq \frac{\ln(k)(4q)^{K/k+1}}{k}+ 1/2
\]
where we additionally assumed that $q \leq k$.
As long as this value is below 1 we can be sure that random greedy coloring strategy succeeds with positive probability.
For $k=K$ we recover the result of \cite{CK15}.
When $K$ is bounded by a linear function of $k$, e.g. $K \leq \alpha k$ it is sufficient that $q$ does not exceed
\[
	\frac{1}{5} \left( \frac{k}{\ln(k)}\right)^{\frac{1}{\alpha+1}}
\]
The resulting bound for $q$ starts to be worse than the one from Theorem \ref{thm:main} when $K$ is roughly of the order $k\log(k)$.

\subsection{Uniform case}
It is instructive to observe how our analysis works for uniform hypergraphs.
We focus on modifications in the proof of our simple bound, since the ideas used for the improved bound do not help in the uniform case.
Using an obvious bound $R_k^e \leq k$, we improve inequality (\ref{eq:XlogY}) to
$
	X \leq \ln(\cp q).
$
Then we apply Lemma \ref{lemma:conv-expt} with $M=\ln(\cp q)$ and $\lambda= \frac{q}{k}$ obtaining
\[
	\probG[\text{$e$ becomes red}] < 2^{-k} \frac{q}{k} \exp(\ln(\cp q))= 2^{-k}  \frac{q^{1+\cp}}{k}.
\]
Since in this case the only bad event that we use is $\mathcal{B}$, we can afford to set $\cp=1+\epsilon$, for any small $\epsilon>0$.
We get that $\mathcal{B}$ does not happen with probability at least $\frac{\epsilon}{1+\epsilon}$.
Then
\[
	\probG[\text{$e$ becomes red}] < 2^{-k} \frac{q^{2+\epsilon}}{k} 
\]
and in order for this value to be at most $\frac{1}{2^k q} \cdot \frac{\epsilon}{1+\epsilon}$ it suffices that
\[
	q \leq k^{\frac{1}{3+\epsilon}} \cdot \left( \frac{\epsilon}{1+\epsilon} \right)^{\frac{1}{3+\epsilon}} .
\] 
This way we obtain a result analogous to that of Beck from \cite{Beck78} (i.e. $m(k) \geq k^{1/3-o(1)}  2^{k} $).
Incorporating the ideas from \cite{RS00} or \cite{CK15} that allowed to derive  a bound $m(k)=\Omega(\sqrt{k/\log(k)}) \cdot  2^k$ does not bring any significant improvement of our main result.

\subsection{Hypergraphs with random-like characteristics}
The weakest points of our analysis are the places where we apply Lemma \ref{lemma:conv-expt}.
The lemma works for any bounded non-negative random variable $X$.
It is clear from the bound that the worst case distribution of $X$ is the one that assumes only values $0$ and $M$.
The variables for which we apply the lemma are related to the numbers of initially monochromatic edges in hypergraphs $H_e$.
If these variables exhibit sufficiently strong concentration around their mean (like in the case of random hypergraphs) we may get much stronger bound than the one of Lemma \ref{lemma:conv-expt} and obtain results that are much closer to the case of uniform hypergaphs.

\bibliographystyle{alpha}
\bibliography{pb}

\end{document}